\documentclass[final]{siamltex}

\usepackage{graphicx}
\usepackage{latexsym}
\usepackage{amssymb}
\usepackage{amsmath}
\usepackage{relsize}
\usepackage{url}
\usepackage{color}
\usepackage{cite}

\def\ub{\bar{u}}
\def\la{\lambda}
\def\xb{\bar{x}}

\def\al{\alpha}
\def\alb{\bar{\al}}
\def\cll{{\mathcal L}}

\newcommand{\para}{\mbox{\rm par}\,}
\newcommand{\aff}{\mbox{\rm aff}\,}

\newtheorem{remark}[theorem]{Remark}
\newtheorem{example}[theorem]{Example}
\newtheorem{claim}[theorem]{Claim}
\newcommand{\nc}{\newcommand}
\nc{\ip}[2]{\mbox{$\langle #1,#2 \rangle$}}
\nc{\linespace}{\vspace{\baselineskip} \noindent}
\nc{\R}{{\bf R}}
\nc{\h}{\R^n}
\nc{\X}{\mathcal{X}}
\nc{\cl}{\mbox{\rm cl}\,}
\nc{\cls}{ \mbox{{\scriptsize {\rm cl}}}\,}
\nc{\conv}{\mbox{\rm conv}}
\nc{\rb}{\mbox{\rm rb}\,}
\nc{\Lip}{\mbox{\rm Lip}}
\nc{\ri}{\mbox{\rm ri}\,}
\nc{\inter}{\mbox{\rm int}\,}
\nc{\kernel}{\mbox{\rm ker}\,}
\nc{\bd}{\mbox{\rm bd}\,}
\nc{\spann}{\mbox{\rm span}\,}
\nc{\rint}{\mbox{\rm rint}\,}
\nc{\epi}{\mbox{\rm epi}\,}
\nc{\gph}{\mbox{\rm gph}\,}
\nc{\rge}{\mbox{\rm rge}\,}
\nc{\rgel}{\mbox{\rm {\scriptsize rge}}\,}
\nc{\sepi}{\mbox{\rm {\scriptsize epi}}\,}
\nc{\sbd}{\mbox{\rm {\scriptsize bd}}\,}
\nc{\dom}{\mbox{\rm dom}\,}
\nc{\sdom}{\mbox{\rm {\scriptsize dom}}\,}
\nc{\sreg}{\mbox{\rm sreg}\,}
\nc{\lin}{\mbox{\rm lin}\,}
\nc{\detr}{\mbox{\rm det}\,}
\nc{\crit}{\mbox{\rm crit}\,}
\nc{\len}{\mbox{\rm length}\,}
\nc{\cone}{\mbox{\rm cone}\,}
\nc{\fix}{\mbox{\rm Fix}}
\nc{\co}{\mbox{\rm co}\,}
\nc{\cco}{\overline{\mbox{\rm co}}\,}
\nc{\lip}{\mbox{\rm lip}\,}
\def\slof{\overline{|\nabla f|}}
\def\sd{\partial}
\def\al{\alpha}

\newcommand{\argmin}{\operatornamewithlimits{argmin}}

\newcommand{\lf}{\operatornamewithlimits{liminf}}
\newcommand{\ls}{\operatornamewithlimits{limsup}}

\title{ Curves of Descent}

\author{D. Drusvyatskiy\thanks{%
    School of Operations Research and Information Engineering,
    Cornell University,
    Ithaca, New York, USA;
    {\tt http://people.orie.cornell.edu/dd379/}.
    Work of Dmitriy Drusvyatskiy on this paper has been partially supported by the NDSEG grant from the Department of Defense and the US-Israel Binational Science Foundation Travel Grant for Young Scientists.
    }%
	\and
	A.D. Ioffe\thanks{
	Department of Mathematics, Technion-Israel Institute of Technology, Haifa, Israel 32000;	
	{\tt http://www.math.technion.ac.il/Site/people/process.php?id=672}.
	Work of A. D. Ioffe was supported in part by the US-Israel Binational Science Foundation Grant 2008261.
	}
	\and	
  A.S. Lewis\thanks{%
  School of Operations Research and Information Engineering,
  Cornell University,
  Ithaca, New York, USA;
  {\tt http://people.orie.cornell.edu/aslewis/}.
  Research supported in part by National Science Foundation Grant DMS-0806057 and by the US-Israel Binational Scientific Foundation Grant 2008261.
}}

\begin{document}

\maketitle

\begin{abstract}
Steepest descent is central in variational mathematics. We present a new transparent existence proof for curves of near-maximal slope --- an influential notion of steepest descent in a nonsmooth setting. We moreover show that for semi-algebraic functions --- prototypical nonpathological functions in nonsmooth optimization --- such curves are precisely the solutions of subgradient dynamical systems. 
\end{abstract}



\pagestyle{myheadings}
\thispagestyle{plain}
\markboth{D. Drusvyatskiy, A. D. Ioffe, and A. S. Lewis}{ Curves of descent}

\section{Introduction}
The intuitive notion of {\em steepest descent} plays a central role in theory and practice. So what are steepest descent curves in an entirely nonsmooth setting? 
To facilitate the discussion, it is useful to introduce notation for ``the fastest instantaneous rate of decrease'' of a function $f$ on $\R^n$, namely the {\em slope} 
$$
|\nabla f|(\bar{x}) :=\ls_{x\to \bar{x}\atop{x\neq \bar{x}}}\frac{(f(\bar{x})-f(x))^+}{\|\bar{x}-x\|}.
$$
Here, we use the convention $r^{+}:=\max\{0,r\}$. The slope of a smooth function simply coincides with the norm of the gradient, and hence the notation. For more details on slope see for example \cite{deg_slope}. Even though the definition is deceptively simple, slope plays a central role in regularity theory and sensitivity analysis; see \cite{ioffe_survey,uni_met}.

One can readily verify for any $1$-Lipschitz curve $\gamma\colon (a,b)\to\R^n$ the upper bound on the speed of descent:
\begin{equation}\label{eqn:steep-dep}
|\nabla (f\circ\gamma)|(t)\leq |\nabla f|(\gamma(t)), \quad\quad \textrm{ for a.e. } t\in (a,b).
\end{equation}
It is then natural to call $\gamma$ a {\em steepest descent curve} if the reverse inequality holds in (\ref{eqn:steep-dep}). Such curves, up to a reparametrization and an integrability condition, are the curves of maximal slope studied in \cite{grad_flows, min_mov,max_slppe,eva_first}. Evidently, the slope is not a lower-semicontinuous function of its argument and hence is highly unstable. Replacing the slope $|\nabla f|$ with its lower-semicontinuous envelope in equation (\ref{eqn:steep-dep}) defines {\em near-steepest descent curves}. See Definition~\ref{defn:steep_des} for a more precise statement.

The question concerning existence of near-steepest descent curves is at the core of the subject. Roughly speaking, there are two strategies in the literature for constructing such curves for a function $f$ on $\R^n$. The first one revolves around minimizing $f$ on an increasing sequence of balls around a point until the radius hits a certain threshold, at which point one moves the center to the next iterate and repeats the procedure. Passing to the limit as the thresholds tend to zero, under suitable conditions and a reparametrization, yields a near-steepest descent curve \cite[Section 4]{max_slppe}. The second approach is based on De Georgi's generalized movements \cite{min_mov}. Namely, one builds a piecewise constant curve by declaring the next iterate to be a minimizer of the function $f$ plus a scaling of the squared distance from the previous iterate \cite[Chapter 2]{grad_flows}. The analysis, in both cases, is highly nontrivial and moreover does not give an intuitive meaning to the parametrization of the curve used in the construction. 

In the current work, we propose an alternate transparent strategy for constructing near-steepest descent curves. The key idea of our construction is to discretize the range of $f$ and then build a piecewise linear curve by projecting iterates onto successive sublevel sets. Passing to the limit as the mesh of the partition tends to zero, under reasonable conditions and a reparametrization, yields a near-steepest descent curve.
Moreover, the parametrization of the curve used in the construction is entirely intuitive: the values of the function parametrize the curve. From a technical viewpoint, this type of a parametrization allows for the deep theory of metric regularity to enter the picture \cite{ioffe_survey,imp}, thereby yielding a simple and elegant existence proof.

The question concerning when solutions of subgradient dynamical systems and curves of maximal slope are one and the same has been studied as well. However a major standing assumption that has so far been needed to establish positive answers in this direction is that the slope of the function $f$ is itself a lower-semicontinuous function \cite{grad_flows,max_slppe} and hence it coincides with the limiting slope --- an assumption that many common functions of nonsmooth optimization (e.g. $f(x)=\min\{x,0\}$) do not satisfy. In the current work, we study this question in absence of such a continuity condition. As a result, {\em semi-algebraic functions} --- those functions whose epigraph can be written as a finite union of sets, each defined by finitely many polynomial inequalities \cite{Coste-semi,DM} --- come to the fore. For semi-algebraic functions that are locally Lipschitz continuous on their domains, solutions of subgradient dynamical systems are one and the same as curves of near-maximal slope. Going a step further, using an argument based on the Kurdyka-{\L}ojasiewicz inequality, in the spirit of \cite{tame_opt,Kur,loja,tailwag}, we show that bounded curves of near-maximal slope for semi-algebraic functions necessarily have finite length. Consequently, such curves defined on maximal domains must converge to a critical point of $f$.

In our writing style, rather than striving for maximal generality, we have tried to make the basic ideas and the techniques as clear as possible. The outline of the manuscript is as follows. Section~\ref{sec:prelim} is a short self-contained treatment of variational analysis in metric spaces. In this section, we emphasize that the slope provides a very precise way of quantifying error bounds (Lemma~\ref{blemma}). In Section~\ref{sec:exist} we prove that curves of near-steepest descent exist under reasonable conditions.
In Section~\ref{sec:comp}, we analyze conditions under which curves of near-steepest descent are the same as solutions to subgradient dynamical systems, with semi-algebraic geometry playing a key role.

\section{Preliminaries: Variational analysis in metric spaces}\label{sec:prelim}
Throughout this section, we will let $(\X,d)$ be a complete metric space. We stress that completeness of the metric space will be essential throughout.
Consider the extended real line $\overline{\R}:=\R\cup\{-\infty\}\cup\{+\infty\}$. We say that an extended-real-valued function is proper if it is never $\{-\infty\}$ and is not always $\{+\infty\}$.  
For a function $f\colon\X\rightarrow\overline{\R}$, the {\em domain} of $f$ is $$\mbox{\rm dom}\, f:=\{x\in\X: f(x)<+\infty\},$$ and the {\em epigraph} of $f$ is $$\mbox{\rm epi}\, f:= \{(x,r)\in\X\times\R: r\geq f(x)\}.$$
A function $f\colon\X\to\overline{\R}$ is {\em lower-semicontinuous} (or {\em lsc} for short) at $\bar{x}$ if the inequality $\lf_{x\to\bar{x}} f(x)\geq f(\bar{x})$ holds. For a set $Q\subset\X$ and a point $x\in\X$, the {\em distance} of $x$ from $Q$ is   $$d(x,Q):=\inf_{y\in Q} d(x,y),$$ and the {\em metric projection} of $x$ onto $Q$ is $$P_Q(x):=\{y\in Q: d(x,y)=d(x,Q)\}.$$

\subsection{Slope and error bounds}
A fundamental notion in local variational analysis is that of {\em slope} --- the ``fastest instantaneous rate of decrease'' of a function. For more details about slope and its relevance to the theory of metric regularity, see \cite{uni_met,ioffe_survey}.
\begin{definition}[Slope]\label{slo}{\rm Consider a function $f\colon\X\to\overline{\R}$, and a point $\bar{x}\in\X$ with $f(\bar{x})$ finite. The {\it slope} of $f$ at $\bar{x}$ is
$$
|\nabla f|(\bar{x}) :=\ls_{x\to \bar{x}\atop{x\neq \bar{x}}}\frac{(f(\bar{x})-f(x))^+}{d(\bar{x},x)}.
$$
The {\it limiting slope} is
$$
\slof(\bar{x}):=\lf_{x\xrightarrow[f]{} \bar{x}}|\nabla f|(x),
$$ where the convergence $x\xrightarrow[f]{} \bar{x}$ means $(x,f(x))\to (\bar{x},f(\bar{x}))$.}
\end{definition}

Slope allows us to define generalized critical points.
\smallskip
\begin{definition}[Lower-critical points]
{\rm 
Consider a function $f\colon\R^n\to\overline{\R}$. We will call any point $\bar{x}$ satisfying $\slof (\bar{x})=0$ a {\em lower-critical point} of $f$.}
\end{definition}
\smallskip

For ${\bf C}^1$-smooth functions $f$ on a Hilbert space, both $\slof(\bar{x})$ and $|\nabla f|(\bar{x})$ simply coincide with the norm of the gradient of $f$ at $\bar{x}$, and hence the notation. In particular, lower-critical points of such functions are critical points in the classical sense.

\smallskip
\begin{proposition}[Slope of a composition]\label{prop:slope_param}
Consider a lsc function $f\colon[a,b]\to\overline{\R}$ and a nondecreasing continuous function $s\colon [c,d]\to [a,b]$. Suppose that $s$ is differentiable at a point $t\in (c,d)$ with $s'(t)\neq 0$.
Then the equality
$$|\nabla (f\circ s)|(t)=|\nabla f|(s(t))\cdot |s'(t)|, \quad\quad \textrm{ holds}.$$
\end{proposition}
\begin{proof}
First, since $s$ is nondecreasing, continuous, and satisfies $s'(t)\neq 0$, we deduce that $s$ is locally open near $t$. Taking this into account, we deduce the chain of equalities
\begin{align*}
|\nabla (f\circ s)|(t)&=\ls_{\tau \to t} \frac{(f(s(t))-f(s(\tau)))^{+}}{|t-\tau|}\\
&=\ls_{\tau\to t} \frac{(f(s(t))-f(s(\tau)))^{+}}{|s(t)-s(\tau)|}\cdot\frac{|s(t)-s(\tau)|}{|t-\tau|}= |\nabla f|(s(t))\cdot |s'(t)|,
\end{align*}
thereby establishing the claimed result.
\end{proof}

\smallskip
We record below the celebrated Ekeland's variational principle.
\smallskip

\begin{theorem}[Ekeland's variational principle]\label{thm:ek} 
Consider a lsc function $g\colon\X\to\overline{\R}$ that is bounded from below. Suppose that for some $\epsilon >0$ and $x\in \R^n$, we have $g(x)\leq\inf f +\epsilon$. Then for any $\rho >0$, there exists a point $\bar{u}$ satisfying 
\begin{itemize}
\item $g(\bar{u})\leq g(x)$,
\item $d(\bar{u},x)\leq \rho^{-1}\epsilon$, and
\item $g(u)+\rho\, d(u,\bar{u}) > g(\bar{u}), \quad\textrm{ for all } u\in \X\setminus\{\bar{u}\}.$
\end{itemize}
\end{theorem}
\bigskip
The following consequence of Ekeland's variational principle will play a crucial role in our work \cite[Basic Lemma, Chapter 1]{ioffe_survey}. We provide a proof for completeness.

\smallskip
\begin{lemma}[Error bound]\label{blemma}
Consider a lsc function $f\colon \X\to\overline{\R}$. Assume that for some point $x\in \dom f$,
there are constants $\al< f(x)$ and $r,K>0$ so that the implication 
$$\al<f(u)\le f(x) \quad\textrm{ and }\quad d(u,x)\le K\quad\Longrightarrow\quad |\nabla f|(u)\ge r, \quad\textrm{ holds}.$$
If in addition the inequality $f(x)-\al<Kr$ is valid, then the sublevel set $[f\le\al]$ is nonempty and we have the estimate
$d(x,[f\le\al])\le r^{-1}(f(x)-\al)$.
\end{lemma}

\proof  Define a lsc function $g\colon\X\to\overline{\R}$ by setting $g(u) := (f(u)-\al)^+$, and choose a real number $\rho<r$ satisfying $f(x)-\al<K\rho$. 
By Ekeland's principle (Theorem~\ref{thm:ek}), there exists a point $\ub$
satisfying 
$$g(\ub)\le g(x),\quad d(\ub,x)\le \rho^{-1}g(x)\le K$$  and $$g(u)+\rho\, d(u,\ub) \ge g(\ub), \quad\textrm{ for all } u.$$
Consequently we obtain the inequality $|\nabla g|(\ub)\le \rho$.
On the other hand, a simple computation shows that this can happen only provided $g(\ub)=0$, for otherwise we would have $|\nabla g|(\ub)=|\nabla f|(\ub)\ge r$.
Hence $\bar{u}$ lies in the level set $[f\le\al]$, and we obtain the estimate $d(x,[f\le\al])\le \rho^{-1}(f(x)-\al)$.
The result now follows by taking $\rho$ arbitrarily close to (and still smaller than) $r$.
\endproof

\subsection{Absolute continuity and the metric derivative}
In this section, we adhere closely to the notation and the development in \cite{grad_flows}.
\smallskip
\begin{definition}[Absolutely continuous curves]
{\rm
Consider a curve $\gamma\colon (a,b)\to \X$. We will say that $\gamma$ is {\em absolutely continuous}, denoted $\gamma\in AC(a,b,\X)$, provided that there exists an integrable function $m\colon (a,b)\to\R$ satisfying 
\begin{equation}\label{eqn:met_dir}
d(\gamma(s),\gamma(t))\leq \int^t_s m(\tau)\; d\tau, \quad \textrm{ whenever } a<s\leq t<b.
\end{equation}}
\end{definition}

Every curve $\gamma\in AC(a,b,\X)$ is uniformly continuous. Moreover, the right and left limits of $\gamma$, denoted respectively by $\gamma(a)$ and $\gamma(b)$, exist. There is a canonical choice for the integrand appearing in the definition of absolute continuity, namely the ``metric derivative''.
\smallskip

\begin{definition}[Metric derivative]
{\rm
For any curve $\gamma\colon [a,b]\to\X$ and any $t\in (a,b)$, the quantity
$$\|\dot{\gamma}(t)\|:=\lim_{s\to t} \frac{d(v(s),v(t))}{|s-t|},$$
if it exists, is the {\em metric derivative} of $\gamma$ at $t$. If this limit does exist at $t$, then we will say that $\gamma$ is {\em metrically differentiable} at $t$.}
\end{definition}
\smallskip

Some comments concerning our notation are in order, since it deviates slightly from that used in the standard monograph on the subject \cite{grad_flows}. The notation $\|\dot{\gamma}(t)\|$ is natural, since whenever $\gamma$ is a differentiable curve into a Hilbert space, the metric derivative is simply the norm of its derivative. This abuse of notation should not cause confusion in what follows.

For any curve $\gamma\in AC(a,b,\X)$, the metric derivative exists almost everywhere on $(a,b)$. Moreover the function $t\mapsto \|\dot{\gamma}(t)\|$ is integrable on $(a,b)$ and is an admissible integrand in inequality (\ref{eqn:met_dir}). In fact, as far as such integrands are concerned, the metric derivative is in a sense minimal. Namely, for any admissible integrand $m\colon (a,b)\to\R$ for the right-hand-side of (\ref{eqn:met_dir}), the inequality
$$\|\dot{\gamma}(t)\|\leq m(t) \quad\textrm{ holds for a.e. } t\in (a,b).$$
See \cite[Theorem 1.1.2]{grad_flows} for more details. We can now define the {\em length} of any absolutely continuous curve $\gamma\in AC(a,b,\X)$ by the formula 
$$\len(\gamma):=\int^b_a  \|\dot{\gamma}(\tau)\|\; d\tau.$$
We adopt the following convention with respect to curve reparametrizations. 
\smallskip
\begin{definition}[Curve reparametrization]
{\rm
Consider a curve $\gamma\colon [a,b]\to\X$. Then any curve $\omega\colon [c,d]\to \X$ is a {\em reparametrization} of $\gamma$ whenever there exists a nondecreasing absolutely continuous function $s\colon [c,d]\to [a,b]$ with $s(c)=a$, $s(d)=b$, and satisfying $\omega=\gamma\circ s$.
}
\end{definition}

\smallskip
Absolutely continuous curves can always be parametrized by arclength. See for example \cite[Lemma 1.1.4]{grad_flows} or \cite[Proposition 2.5.9]{met_geo}.
\smallskip

\begin{theorem}[Arclength parametrization]\label{thm:arc_param}
Consider an absolutely continuous curve $\gamma\in AC(a,b,\X)$, and denote its length by $L=\len(\gamma)$. Then there exists a nondecreasing absolutely continuous map $s\colon [a,b]\to [0,L]$ with $s(a)=0$ and $s(b)=L$, and a 1-Lipschitz curve $v\colon [0,L]\to\X$ satisfying 
$$\gamma=v\circ s \quad \textrm{ and }\quad  \|\dot{v}\|=1 \textrm{ a.e. in } [0,L].$$
\end{theorem}
\begin{proposition}[Metric derivative of a composition]\label{prop:chain}{\\}
Consider a curve $\gamma\colon [0,L]\to \X$ and a continuous function $s\colon [a,b]\to [0,L]$. Consider a point $t\in (a,b)$, so that $s$ is differentiable at $t$ and $\gamma$ is metrically differentiable at $s(t)$. Then the curve $\gamma\circ s$ is metrically differentiable at $t$ with metric derivative $\|\dot{\gamma}(s(t))\|\cdot |s'(t)|$.
\end{proposition}
\begin{proof}
Observe that for any sequence of points $t_i\to t$ with $t_i\neq t$ and $s(t_i)=s(t)$ for each index $i$, we have
$$\lim_{i\to\infty} \frac{d(\gamma(s(t_i)),\gamma(s(t)))}{|t_i-t|}=0=\|\dot{\gamma}(s(t))\|\cdot |s'(t)|.$$
On the other hand, for any sequence $t_i\to t$ with $t_i\neq t$ and $s(t_i)\neq s(t)$ for each index $i$, we have
$$\lim_{i\to\infty} \frac{d(\gamma(s(t_i)),\gamma(s(t)))}{|t_i-t|}=\lim_{i\to\infty} \frac{d(\gamma(s(t_i)),\gamma(s(t)))}{|s(t_i)-s(t)|}\cdot\frac{|s(t_i)-s(t)|}{|t_i-t|}=\|\dot{\gamma}(s(t))\|\cdot |s'(t)|.$$ The result follows.
\end{proof}

\smallskip
The following is a Sard type theorem for real-valued functions of one variable. See for example \cite[Fundamental Lemma]{absol_sard}.
\smallskip

\begin{theorem}[Sard theorem for functions of one variable]\label{thm_sard}
For any function $s\colon[a,b]\to\R$, the set  
$$\{t\in [a,b]: \exists \tau\in s^{-1}(t)\textrm{ with } s'(t)=0\}$$
is Lebesgue null.
\end{theorem}

\smallskip
The following theorem provides a convenient way of determining when strictly monotone, continuous functions are absolutely continuous \cite{sob}.
\smallskip
\begin{theorem}[Inverses of absolutely continuous functions]\label{thm:inv_abs}
Consider a continuous, strictly increasing function $s\colon [a,b]\to \R$. Then the inverse $s^{-1}\colon [s(a),s(b)]\to [a,b]$ is absolutely continuous if and only if the set 
$$E:=\{t\in [a,b]: s'(t)=0\},$$
has Lebesgue measure zero.
\end{theorem}

\subsection{Steepest descent and gradient flow}
In this section we consider steepest descent curves in a purely metric setting. To this end, consider a lsc function $f\colon\X\to\R$ and a 1-Lipschitz continuous curve $\gamma\colon [a,b]\to\X$. There are two intuitive requirements that we would like $\gamma$ to satisfy in order to be called a steepest descent curve:
\begin{enumerate}
\item The composition $f\circ\gamma$ is non-increasing,
\item The instantaneous rate of decrease of $f\circ\gamma$ is almost always as great as possible.
\end{enumerate}
To elaborate on the latter requirement, suppose that the composition $f\circ\gamma$ is indeed non-increasing. 
Then taking into account that $\gamma$ is 1-Lipschitz and that monotone functions are differentiable a.e., one can readily verify 
\begin{equation}\label{eqn:descent}
|(f\circ\gamma)'(t)|\leq |\nabla f|(\gamma(t)), \quad\quad \textrm{ for a.e. } t\in (a,b).
\end{equation}
Requiring the inverse inequality to hold amounts to forcing the curve to achieve fastest instantaneous rate of decrease. The discussion above motivates the following definition.
\smallskip

\begin{definition}[Near-steepest descent curves]\label{defn:steep_des}
{\rm
Consider a lsc function $f\colon\X\to\overline{\R}$. Then a 1-Lipschitz continuous curve $\gamma\colon [a,b]\to\X$ is a {\em steepest descent curve} provided that $f\circ \gamma$ is non-increasing and the inequality
$$|(f\circ\gamma)'(t)|\geq |\nabla f|(\gamma(t)),\quad\quad\textrm{ holds for a.e. } t\in [a,b].$$ 
If instead the weaker inequality 
$$|(f\circ\gamma)'(t)|\geq \slof(\gamma(t)),\quad\quad\textrm{ holds for a.e. } t\in [a,b],$$ 
then we will say that $\gamma$ is a {\em near-steepest descent curve}.
}
\end{definition}
\smallskip

In principle, near-steepest descent curves may fall short of achieving true ``steepest descent'', since the analogue of inequality (\ref{eqn:descent}) for the limiting slope may fail to hold in general. Our work, however, will revolve around near-steepest descent curves since the limiting slope is a much better behaved object, and anyway this is common practice in the literature (see for example \cite{grad_flows,eva_first,max_slppe}). The following example illustrates the difference between the two notions.

\smallskip
\begin{example}[Steepest descent vs. near-steepest descent]\label{exa:vs}
{\rm
Consider the function $f\colon\R^2\to\R$ defined by $f(x,y):=-x+\min(y,0)$. Then the curve $x(t)=(t,0)$ is a near-steepest descent curve but is not a steepest descent curve, as one can easily verify.
}
\end{example}
\smallskip

It is often convenient to reparametrise near-steepest descent curves so that their speed is given by the slope. This motivates the following companion notion; related concepts appear in \cite[Section 1.3]{grad_flows},\cite{max_slppe,lec_sav}.


\smallskip
\begin{definition}[Curve of near-maximal slope] \label{defn:cms}
{\rm
Consider a lsc function $f\colon\X\to\overline{\R}$. A curve $\gamma\colon [a,b]\to\X$ is a {\em curve of near-maximal slope} if the following conditions holds:

(a) \ $\gamma$ is absolutely continuous,

(b) \  $\| \dot \gamma(t)\|= \slof(\gamma(t))$ a.e. on $[a,b]$,

(c) \ $f\circ \gamma$ is nonincreasing and satisfies
$$
(f\circ \gamma)'(t)\leq - \big(\slof(\gamma(t))\big)^2 \quad   {\rm a.e. \  on} \;  [a,b].
$$}
\end{definition}
The following proposition shows that, as alluded to above, near-steepest descent curves and curves of near-maximal slope are the same up to reparametrization, provided that a minor integrability condition is satisfied. 
\smallskip
\begin{proposition}[Curves of near-steepest descent \& near-maximal slope]\label{prop:trans}
Consider a lsc function $f\colon\X\to\overline{\R}$ and a near-steepest descent curve $\gamma\colon [a,b]\to \X$. If in addition 
$$(\slof \circ \gamma)^{-1} \textrm{ is integrable and } \slof \circ \gamma \textrm{ is finite a.e. on } [a,b],$$ then there exists a reparametrization of $\gamma$ that is a curve of near-maximal slope. 

Conversely, consider a curve of near-maximal slope $\gamma\colon [a,b]\to \X$. If in addition
$$\slof \circ \gamma \textrm{ is integrable and } (\slof \circ \gamma)^{-1} \textrm{ is finite a.e. on } [a,b],$$ then there exists a reparametrization of $\gamma$ that is a near-steepest descent curve. 
\end{proposition}
\begin{proof}
To see the validity of the first claim,  let $\eta:=\int^b_a (\slof \circ \gamma)(t) \;dt$ and define the function $s\colon [a,b]\to [0,\eta]$ by setting
$$s(t):=\int^t_a \frac{1}{\slof(\gamma(r))} \;dr.$$ Then $s$ is a strictly increasing, absolutely continuous function.
Moreover, by Theorem~\ref{thm:inv_abs}, the inverse $s^{-1}:[0,\eta]\to [a,b]$ is absolutely continuous as well.
Define now the function $\omega\colon [0,\eta]\to [a,b]$ by setting $\omega(\tau):=\gamma(s^{-1}(\tau)).$ Clearly $\omega$ is absolutely continuous, and using Propositions~\ref{prop:chain} and \ref{prop:slope_param} we deduce
$$\|\dot{\omega}(\tau)\|=\slof(\omega(\tau)) \quad \textrm{and}\quad  (f\circ \omega)'(t)\leq - \big(\slof(\omega(t))\big)^2 \quad   {\rm a.e. \  on} \;  [a,b].$$
This shows that $\omega$ is a curve of near-maximal slope. The converse claim follows by similar means.
\end{proof}

\section{Existence of descent curves}\label{sec:exist}
In this section, we provide a natural and transparent existence proof for near-steepest descent curves in complete locally convex metric spaces. We begin with a few relevant definitions, adhering closely to the notation of \cite{fix_met}.
\smallskip
\begin{definition}[Metric segments]
{\em  \\
A subset $S$ of a  metric space $\mathcal{X}$ is a {\em metric segment} between two points $x$ and $y$ in $\mathcal{X}$ if there exists a closed interval $[a,b]$ and an isometry $\omega\colon [a,b]\to\mathcal{X}$ satisfying $\omega([a,b])=S$, $\omega(a)=x$, and $\omega(b)=y$. 
}
\end{definition}
\smallskip

\begin{definition}[Convex metric spaces]
{\rm
We will say that $\mathcal{X}$ is a {\em convex metric space} if for any distinct points $x,y\in \mathcal{X}$ there exists a metric segment between them. We will call $\mathcal{X}$ a {\em locally convex metric space} if each point in $\X$ admits a neighborhood that is a convex metric space in the induced metric.}
\end{definition}
\smallskip

Some notable examples of locally convex metric spaces are complete Riemannian manifolds and, more generally, length spaces that are complete and locally compact (see Hopf-Rinow theorem). For more examples, we refer the reader to \cite{met_grom}. 

We now introduce the following very weak continuity condition, which has been essential in the study of descent curves in metric spaces. See for example \cite[Theorem 2.3.1]{grad_flows}.
\begin{definition}[Continuity on slope-bounded sets]
{\rm
Consider a function $f\colon\X\to\overline{\R}$. We will say that $f$ is {\em continuous on slope-bounded sets} provided that for any point $\bar{x}\in\dom f$ the implication 
$$x_i\to\bar{x} \quad\textrm{with}\quad \sup_{i\in\mathbb{N}}\{|\nabla f|(x_i),d(x_i,\bar{x}),f(x_i)\}<\infty\quad \quad \Longrightarrow \quad\quad f(x_i)\to f(\bar{x}),$$
holds.}
\end{definition}



\smallskip


We now arrive at the main result of this section. We should note that in the following theorem we will suppose tough compactness assumptions relative to the metric topology. As is now standard, such compactness assumptions can be sidestepped by instead introducing weaker topologies \cite[Section 2.1]{grad_flows}. On the other hand, following this route would take us far off field and would lead to technical details that may obscure the main proof ideas for the reader. Hence we do not dwell on this issue further. We have however designed our proof so as to make such an extension as easy as possible for interested readers.
\smallskip
\begin{theorem}[Existence of near-steepest descent curves]\label{thm:near_steep}
Consider a lsc function $f\colon\X\to\overline{\R}$ on a complete locally convex metric space $\mathcal{X}$, along with a point $\bar{x}$ in the domain of $f$. Suppose that $f$ is continuous on slope-bounded sets and that bounded closed subsets of sublevel sets of $f$ are compact. Then there exists a curve $\gamma\colon [0,L]\to\X$ emanating from $\bar{x}$ and satisfying the following properties.
\begin{description}
\item[Decrease in value:] The function $f\circ\gamma$ is nonincreasing.
\item[Near-steepest descent:] $\gamma$ is 1-Lipschitz continuous and satisfies $$|(f\circ\gamma)'(t)|\geq \slof(\gamma(t)),\quad \textrm{ for a.e. }t\in[0,L].$$
\item[Regularity:] The function $\slof(\gamma(\cdot))$ is integrable on $[a,b]$ and we have $\|\dot{\gamma}(t)\|=1$ for a.e. $t\in [0,L]$.
\end{description}
\end{theorem}
\begin{proof}
First, by restricting attention to a sufficiently small neighborhood of $\bar{x}$ we can clearly assume that $\X$ is a convex metric space.
If the equality $\slof (\bar{x})=0$ were to hold, then the constant curve $\gamma\equiv \bar{x}$ would satisfy all the required properties. Hence we may suppose that $\bar{x}$ is not a lower-critical point of $f$. We can then find constants $\eta>0$, $r> 0$ and $C>0$ so that all conditions of Lemma \ref{blemma}
are satisfied for $\xb$,  $\al=f(\xb)-\eta$ and $K=C$. In particular, the level set $[f\leq f(\xb)-\eta]$ is nonempty. Shrinking $\eta$ we may enforce the inequality $\eta<rC$. Let $0=\tau_0<\tau_1<\ldots<\tau_k=\eta$ be a partition of $[0,\eta]$ 
into $k$ equal parts. We will adopt the notation
$$\lambda:=\frac{\tau_{i+1}-\tau_i}{\eta}, \quad\quad \alpha_i=f(\bar{x})-\tau_i, \quad\quad L_i:=[f\leq \alpha_i].$$ 
With this partition, we will associate a certain curve $u_k(\tau)$ for $\tau\in[0,\eta]$, naturally obtained by concatenating metric segments between points $x_i$, $x_{i+1}$ lying on consecutive sublevel sets. See Figure 3.1 for an illustration. For notational convenience, we will often suppress the index $k$ in $u_k(\tau)$. The construction is as follows. Set $u(0)=\xb$, and suppose that we have defined points $x_i$ for $i=0,\ldots,j$. We will now inductively define $x_{j+1}$. To this end, consider the quantity $$r_j:=\inf\,\{|\nabla f|(y):\; 
\alpha_{j+1}<f(y)\le f(x_j), ~d(y,x_j) <\la C\}.$$ 
and let $x_{j+1}$ be any point satisfying
$$x_{j+1}\in L_{j+1}\quad\quad \textrm{and} \quad\quad d(x_{j+1}, x_{j})\leq r_{j}^{-1}(f(x_j)-\alpha_{j+1})^{+}.$$
(In our setting, due to the compactness of bounded closed subsets of sublevel sets of $f$, we may simply define $x_{j+1}$ to be any closest point of $L_{j+1}$ to $x_j$.) 
\smallskip
\begin{claim}[Well-definedness]
For all indices $i=0,\ldots,k$, the points $x_i$ are well defined and satisfy
\begin{equation}\label{eq:lip_met}
d(x_{i+1},x_i)\le r_i^{-1}(\tau_{i+1}-\tau_i),
\end{equation}
and
\begin{equation}\label{step_met}
r_{i+1}\geq r,\quad d(x_{i+1},\xb)\le r^{-1}\tau_{i+1}.
\end{equation}
\end{claim}

\begin{proof}
The proof proceeds by induction. To this end, suppose that the points $x_i$ are well defined for indices $i=0,\ldots,j$ and the equations (\ref{eq:lip_met}) and (\ref{step_met}) are valid for indices $i=0,\ldots,j-1$. 
Observe if the inequality $f(x_j)\leq \alpha_{j+1}$ were true, then we may set $x_{j+1}:=x_j$ and the inductive step would be true trivially. Hence suppose otherwise. We claim that the conditions of  Lemma~\ref{blemma} are satisfied with 
$x=x_j$, $\al = \alpha_{j+1}$, $K=\la C$, and with $r_j$ in place of $r$.

To this end, we show the following


$\bullet$ \  $f(x_j)-\alpha_{j+1}\le\la r_jC$;

$\bullet$  \  $\alpha_{j+1}<f(y)\le f(x_j) \textrm{ and } d(y,x_j)\le \la C\quad\Longrightarrow\quad |\nabla f|(y)\ge r_j$. 

\noindent 
Observe
$$
f(x_j)-(f(\xb)-\tau_{j+1})\leq\tau_{j+1}-\tau_j\le (\tau_{j+1}-\tau_j)\frac{rC}{\eta}=\la rC\le \la r_j C,
$$
which is the first of the desired relations.  The second relation follows immediately from the definition of $r_j$.

Applying Lemma~\ref{blemma}, we conclude that the point $x_{j+1}$ is well-defined and the inequality 
$d(x_{j+1},x_j)\le r_j^{-1}(f(x_{j})-\alpha_{j+1})^{+}\le r_j^{-1}(\tau_{j+1}-\tau_j)$ holds. Consequently, we obtain
$$d(x_{j+1},\xb)\leq d(x_{j+1},x_j)+d(x_j,\xb)\leq r_j^{-1}(\tau_{j+1}-\tau_j) +r^{-1}\tau_{j}\leq r^{-1}\tau_{j+1}.$$
Finally we claim that the inequality $r_{j+1}\geq r$ holds. To see this, consider a point $y$ satisfying 
$f(\xb)-\tau_{j+2}<f(y)\le  f(x_{j+1})$ and $d(y,x_{j+1})<\la C$. 
Taking (\ref{step_met}) into account, along with the inequality
$r^{-1}\le C/\eta$, we obtain
$$
d(y,\xb)\le d(y,x_{j+1})+ d(x_{j+1},\xb)\le \frac{\tau_{j+2}-\tau_{j+1}}{\eta}C+\frac{\tau_{j+1}}{r}
=\frac{\tau_{j+2}}{\eta}C<C.
$$
Combining this with the obvious inequality
$f(\xb)>f(y)>f(\xb)-\eta$, we deduce $|\nabla f|(y)\ge r$ and consequently $r_{j+1} \geq r$. This completes the induction.
\end{proof}
\smallskip


For each index $i=0,\ldots,k-1$, let $\omega_i\colon [0,d(x_i,x_{i+1})]\to\mathcal{X}$ be the isometry parametrizing the metric segment between $x_i$ and $x_{i+1}$. For reasons which will become apparent momentarily, we now rescale the domain of $\omega_i$ by instead declaring
$$\omega_i\colon [\tau_{i},\tau_{i+1}]\to\mathcal{X}\quad \textrm{ to be }\quad \omega_i(t)=\omega_i\Big(\frac{d(x_{i+1},x_{i})}{\tau_{i+1}-\tau_i}(t-\tau_i)\Big).$$
Observe now that for any $s,t\in [\tau_{i},\tau_{i+1}]$ with $s <t$, we have 
\begin{equation}\label{eqn:lip_new}
d(\omega_i(t),\omega_i(s))=\frac{d(x_{i+1},x_{i})}{\tau_{i+1}-\tau_i}(t-s)\leq r_i^{-1}(t-s).
\end{equation}
It follows that all the curves $\omega_i$ are Lipschitz continuous with a uniform modulus $r^{-1}$. We may now define a curve $u_k\colon [0,\eta]\to\mathcal{X}$ by simply concatenating the domains of $\omega_i$ for each index $i=0,\ldots,k-1$. Clearly the mappings $u_k(\tau)$ are Lipschitz continuous with a uniform modulus $r^{-1}$ (independent of the index $k$). As all of these mappings coincide at $\tau=0$, and bounded subsets of sublevel sets of $f$ are compact, the well-known theorem of Arzel\`{a} and Ascoli (\cite[Section 7]{gen_top}) guarantees that  
a certain subsequence of $u_k(\tau)$ converges uniformly on $[0,\eta]$ to some mapping $x(\tau)$. Furthermore any uniform limit of $u_k(\tau)$ is clearly also Lipschitz continuous with the same constant $r^{-1}$. Observe that the metric derivative functions $\|\dot{u}_k(\cdot)\|$ lie in a bounded subset of $L^{2}(a,b)$. Hence up to a subsequence $\|\dot{u}_k(\cdot)\|$ converge weakly to some integrable mapping $m\colon [0,\eta]\to\R$ satisfying 
\begin{equation}
d(x(s),x(t))\leq \int^t_s m(\tau)\; d\tau, \quad \textrm{ whenever } 0<s\leq t<\eta.
\end{equation}
For what follows now, define the set of breakpoints  $$E:=\bigcup_{k\in\mathbb{N}}\bigcup_{i\in\mathbb{N}\cap[0,k]} \Big\{\frac{i\lambda_k}{\eta}\Big\}$$ and observe that is has zero measure in $[0,\eta]$. In addition, let $D$ be the full-measure subset of $[0,\eta]$ on which all the curves $u_k$ and $x$ admit a metric derivative. 

\smallskip
\begin{claim}\label{cl:prop}
For almost every $\tau\in [0,\eta]$ with $\|\dot{x}(\tau)\|\neq 0$, the following are true: 
\begin{itemize}
\item $f(x(\tau))=f(\bar{x})-\tau$,
\item $\|\dot{x}(\tau)\|\le\frac{1}{\slof(x(\tau))}$,
\item $|\nabla (f\circ x)|(\tau)\geq 1$.
\end{itemize}

\end{claim}
\begin{proof}
Fix a real $\tau\in D\setminus E$ with $\|\dot{x}(\tau)\|\neq 0$. Then using equation (\ref{eqn:lip_new}) we 
deduce 
\begin{equation}\label{eqn:2_met}
\|\dot{u}_k(\tau)\|\leq \frac{1}{r^{(k)}_{i_k}},
\end{equation}
for some $i_k\in \{0,\ldots,k\}$, where the superscript $(k)$ refers to partition of the interval $[0,\eta]$ into $k$ equal pieces. Noting that weak convergence does not increase the norm and using minimality of the metric derivative, we deduce
\begin{equation}\label{eqn:min_met}
\lf_{k\to\infty} \|\dot{u}_k(\tau)\|\geq m(\tau)\geq \|\dot{x}(\tau)\|, \quad \textrm{ for a.e. } \tau\in [0,\eta].
\end{equation}
Consequently there exists a subsequence of $\|\dot{u}_k(\tau)\|$, which we continue to denote by $\|\dot{u}_k(\tau)\|$, satisfying 
$\lim_{k\to\infty} \|\dot{u}_k(\tau)\|\neq 0$. Taking into account (\ref{eqn:2_met}), we deduce that $r^{(k)}_{i_k}$ remain bounded. We may then choose points $x^{(k)}_{i_k}$, $y_k$, and reals $\lambda_k$, $\tau^{(k)}_{i_k}$ with $\tau\in (\tau^{(k)}_{i_k},\tau^{(k)}_{i_k+1})$  satisfying
$$d(y_k,x^{(k)}_{i_k})< \lambda_k C,\quad\quad f(\bar{x})-\tau^{(k)}_{i_k+1}< f(y_k)\leq f(x^{(k)}_{i_k}), \quad x^{(k)}_{i_k}\to x(\tau),\quad \tau^{(k)}_{i_k}\to\tau$$ and $|\nabla f| (y_k)\leq r^{(k)}_{i_k}+\frac{1}{k}$. Then since $f$ is continuous on slope-bounded sets and the quantity $f(x^{(k)}_{i_k})-(f(\bar{x})-\tau^{(k)}_{i_k+1})$ tends to zero, we deduce 
$$f(x(\tau))=\lim_{k\to\infty} f(y_k)=\lim_{k\to\infty} f(\bar{x})-\tau^{(k)}_{i_k+1}=f(\bar{x})-\tau,$$
as claimed. Moreover $$\lf_{k\to\infty}\, r^{(k)}_{i_k} \geq \lf_{k\to\infty}\, \{|\nabla f| (y_k)-\frac{1}{k}\}\geq \slof (x(\tau)).$$
Combining this with (\ref{eqn:min_met}) and taking the limit in (\ref{eqn:2_met}), we obtain
$$\|\dot{x}(\tau)\|\le\frac{1}{\slof(x(\tau))},$$
as claimed.

Consider now a real $\tau\in [0,\eta]$ satisfying $\|\dot{x}(\tau)\|\neq 0$ and $|\nabla (f\circ x)|(\tau)< 1$. Then there exists a neighborhood of $\tau$ in $[0,\eta]$ along with a full-measure subneighborhood on which the metric derivative $\|\dot{x}(\cdot)\|$ is identically zero. It easily follows that the set of such points $\tau$ has zero measure in $[0,\eta]$.
\end{proof}
\smallskip

In particular, it follows from Claim~\ref{cl:prop} that for almost every $\tau\in [0,\eta]$ the implication
$$\|\dot{x}(\tau)\|\neq 0 \quad \Longrightarrow \quad  \slof (x(\tau))<\infty\quad \textrm{ and }\quad |\nabla (f\circ x)|(\tau) \geq \slof (x(\tau))\cdot \|\dot{x}(\tau)\|,$$
holds.

Now in light of Theorem~\ref{thm:arc_param}, there exists a nondecreasing absolutely continuous map $s\colon [a,b]\to [0,L]$ with $s(a)=0$ and $s(b)=L$, and a 1-Lipschitz curve $\gamma\colon [0,L]\to\X$ satisfying
$$x(\tau)=(\gamma\circ s)(\tau) \quad \textrm{ and }\quad  \|\dot{\gamma}(t)\|=1 \textrm{  for a.e. } t\in[0,L].$$
Then by Propositions~\ref{prop:slope_param} and~\ref{prop:chain}, whenever $s$ is differentiable at $\tau$ with $s'(\tau)\neq 0$ and $\gamma$ is metrically differentiable at $s(\tau)$ with $\|\dot{\gamma}(s(\tau))\|=1$, we have
$$\|\dot{x}(\tau)\|=s'(\tau),$$
and 
$$\slof (x(\tau))\cdot \|\dot{x}(\tau)\|\leq|\nabla (f\circ x)|(\tau)=|\nabla( f\circ\gamma)|(s(\tau))\cdot \|\dot{x}(\tau)\|.$$
Moreover it easily follows (in part using Theorem~\ref{thm_sard}) that for a.e. $t\in [0,L]$ and any $\tau\in s^{-1}(t)$ the function $s$ is differentiable at $\tau$ with $s'(\tau)\neq 0$ and $\gamma$ is metrically differentiable at $s(\tau)$ with $\|\dot{\gamma}(s(\tau))\|=1$.
Thus for a.e. $t\in [0,L]$, we have
$$\slof(\gamma(t))<\infty \quad \textrm{and} \quad |\nabla (f\circ\gamma)|(t)\geq \slof(\gamma(t)),$$
as claimed. Finally Claim~\ref{cl:prop}, along with lower-semicontinuity of $f$, easily implies that $f\circ\gamma$ is nonincreasing.

Observe by the same claim that for almost every $t\in [0,L]$ we have 
$$\slof(\gamma(t))\leq \frac{1}{s'(s^{-1}(t)) },$$
where $s^{-1}(t)$ is a singleton and $s'(\tau)\neq 0$.
Define $$E=\{t\in [0,L]: \exists \tau\in s^{-1}(t)\textrm{ with } s'(\tau)=0\}.$$
We obtain
$$\int_{[0,L]}\slof(\gamma(t))\; dt\leq \int_{E^c} \frac{1}{s'(s^{-1}(t)) }dt \leq \eta,$$
thereby completing the proof.
\end{proof}

\begin{figure}[htbp]\label{fig:ill}
   	\centering
   	\includegraphics[scale=0.4]{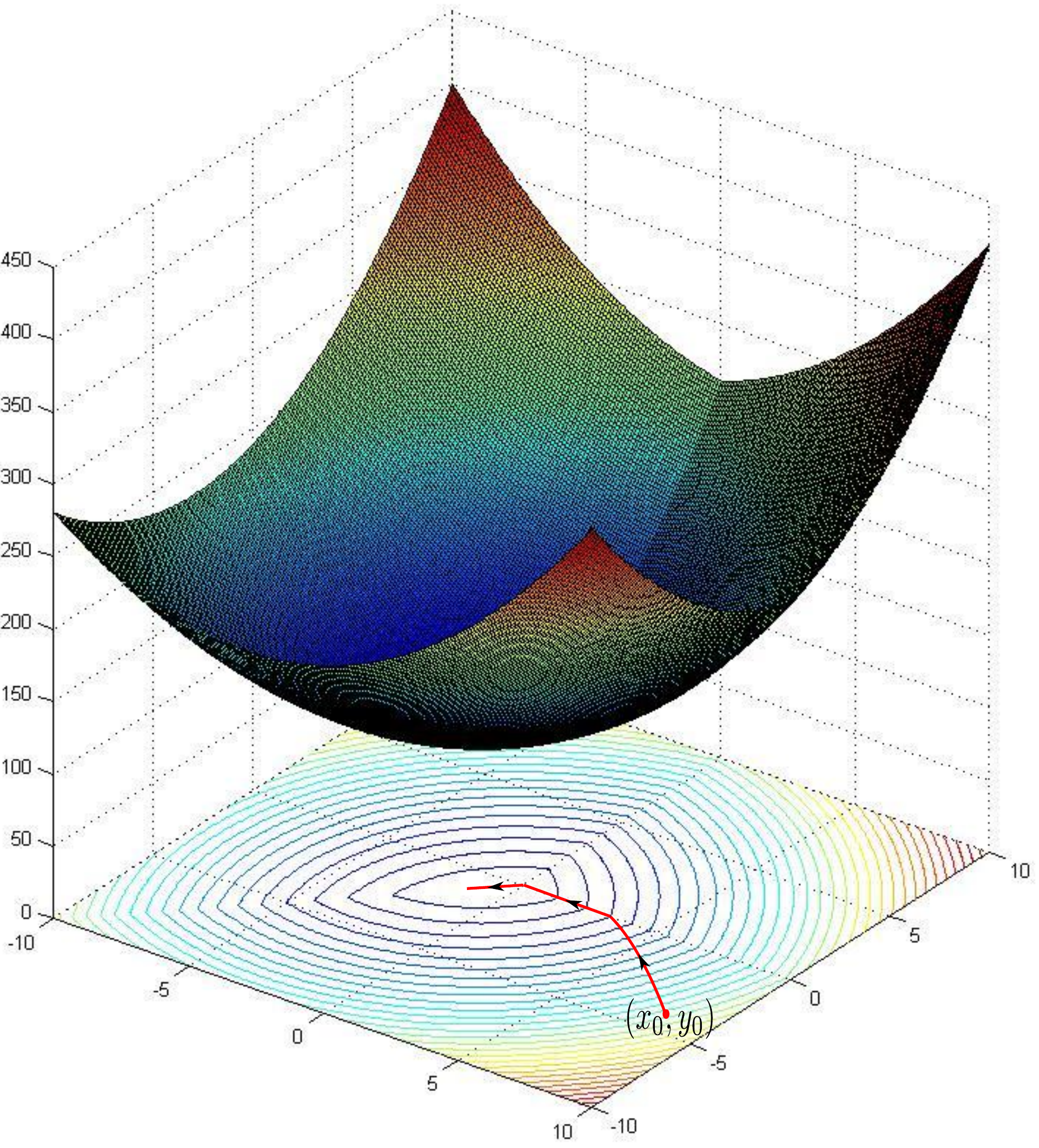}%
	\caption{$f(x,y)=\max\{x+y,\,|x-y|\}+x(x+1)+y(y+1)+100$}	
	\end{figure}

\bigskip

The following is now an easy consequence.
\smallskip
\begin{corollary}[Existence of curves of near-maximal slope]
Consider a lsc function $f\colon\X\to\overline{\R}$ on a complete locally convex metric space $\X$ and a point $\bar{x}\in\X$, with $f$ finite at $\bar{x}$. Suppose that $f$ is continuous on slope-bounded sets and that bounded closed subsets of sublevel sets of $f$ are compact. Then there exists a curve of near-maximal slope $\gamma\colon [0,T]\to\X$ starting at $\bar{x}$.
\end{corollary}
\begin{proof}
This is immediate from Proposition~\ref{prop:trans} and Theorem~\ref{thm:near_steep}.
\end{proof}

\smallskip
\begin{remark}[Steepest descent vs. near-steepest descent] {\\}
{\rm
In Example~\ref{exa:vs} we illustrated a near-steepest descent curve that fails to be a steepest descent curve. On the other hand, the construction used in Theorem~\ref{thm:near_steep} would not produce such a curve. This leads us to conjecture that our construction would always yield a true steepest descent curve at least when applied to semi-algebraic functions. This class of functions is the focal point of the following section.}
\end{remark}

\section{Descent curves in Euclidean spaces}\label{sec:comp}
In this section, we compare curves of near maximal slope to a more classical idea --- solutions of gradient-like dynamical systems. To do so we recall a notion of a generalized gradient, which in principle makes sense in Hilbert spaces. However, since the main results of this section concern semi-algebraic functions --- inherently finite dimensional objects  --- we stay within the setting of Euclidean spaces throughout the section.

\subsection{Some elements of Variational Analysis}\label{sec:prel}
In this section, we summarize some of the fundamental tools used in variational analysis and nonsmooth optimization.
We refer the reader to the monographs of Borwein-Zhu \cite{Borwein-Zhu}, Clarke-Ledyaev-Stern-Wolenski \cite{CLSW},
Mordukhovich \cite{Mord_1}, Penot \cite{penot_book}, and Rockafellar-Wets \cite{VA}, and to the survey of Ioffe \cite{ioffe_survey}, for more details. Unless otherwise stated, we follow the terminology and notation of \cite{ioffe_survey} and \cite{VA}.


Throughout this section, we will consider a real Euclidean space $\h$ with inner product $\langle \cdot,\cdot \rangle$. The symbol $\|\cdot\|$ will denote the corresponding norm on $\h$. 
Henceforth, the symbol $o(\|x-\bar{x}\|)$ will denote a term with the property 
$$\frac{o(\|x-\bar{x}\|)}{\|x-\bar{x}\|}\rightarrow 0,\quad \textrm{ when }x\to \bar{x} \textrm{ with }x\neq\bar{x}.$$
The symbols  $\mbox{\rm cl}\, Q$, $\mbox{\rm conv}\, Q$, $\cone Q$, and $\aff Q$ will denote the topological closure, the convex hull, the (non-convex) conical hull, and the affine span of $Q$ respectively. The symbol $\para Q$ will denote the parallel subspace of $Q$, namely the set $\para Q:=\aff Q - \aff Q$. 
An open ball of radius $\epsilon$ around a point $\bar{x}$ will be denoted by $B_{\epsilon}(\bar{x})$, while the open unit ball will be denoted by ${\bf B}$. A primary variational-analytic method for studying nonsmooth functions on $\h$ is by means of subdifferentials.
\smallskip
\begin{definition}[Subdifferentials]
{\rm Consider a function $f\colon\h\to\overline{\R}$ and a point $\bar{x}$ with $f(\bar{x})$ finite. 
\begin{enumerate}
\item The {\em Fr\'{e}chet subdifferential} of $f$ at $\bar{x}$, denoted 
$\hat{\partial}f(\bar{x})$, consists of all vectors $v \in \h$ satisfying $$f(x)\geq f(\bar{x})+\langle v,x-\bar{x} \rangle +o(\|x-\bar{x}\|).$$ 
\item The {\em limiting subdifferential} of $f$ at $\bar{x}$, denoted $\partial f(\bar{x})$, consists of all vectors $v\in\h$ for which there exist sequences $x_i\in\R^n$ and $v_i\in\hat{\partial} f(x_i)$ with $(x_i,f(x_i),v_i)$ converging to $(\bar{x},f(\bar{x}),v)$.
\item The {\em horizon subdifferential} of $f$ at $\bar{x}$, 
denoted $\partial^{\infty} f(\bar{x})$, consists of all vectors $v\in\h$ for which there exists a sequence of real numbers $\tau_i\downarrow 0$ and a sequence of points $x_i\in\h$, along with subgradients $v_i\in\hat{\partial} f(x_i)$, so that $(x_i,f(x_i),\tau_i v_i )$ converge to $(\bar{x},f(\bar{x}),v)$.
\item The {\em Clarke subdifferential} of $f$ at $\bar{x}$, denoted $\partial_c f(\bar{x})$, is obtained by the convexification 
$$\partial_c f(\bar{x}):=\cl\co [\partial f(\bar{x})+\partial^{\infty} f(\bar{x})].$$
\end{enumerate}
We say that $f$ is {\em subdifferentiable} at $\bar{x}$ whenever $\partial f(\bar{x})$ is nonempty (equivalently when $\partial_c f(\bar{x})$ is nonempty).
} 
\end{definition}
\smallskip

In particular, every locally Lipschitz continuous function is subdifferentiable. For $x$ such that $f(x)$ is not finite, we follow the convention that $\hat{\partial}f(x)=\partial f(x)=\partial^{\infty} f(x)=\partial_c f(\bar{x})=\emptyset$. 

The subdifferentials $\hat{\partial} f(\bar{x})$, $\partial f(\bar{x})$, and $\partial_c f(\bar{x})$ generalize the classical notion of gradient. In particular, for ${\bf C}^1$-smooth functions $f$ on $\h$, these three subdifferentials consist only of the gradient $\nabla f(x)$ for each $x\in\h$. For convex $f$, these subdifferentials coincide with the convex subdifferential. The horizon subdifferential $\partial^{\infty} f(\bar{x})$ plays an entirely different role; namely, it detects horizontal ``normals'' to the epigraph. In particular, a lsc function $f\colon\h\to\overline{\R}$ is locally Lipschitz continuous around $\bar{x}$ if and only if we have $\partial^{\infty} f(\bar{x})=\{0\}$.

For a set $Q\subset\h$, we define the {\em indicator function} of $Q$, denoted $\delta_Q$, to be zero on $Q$ and plus infinity elsewhere. The geometric counterparts of subdifferentials are normal cones.
\smallskip
\begin{definition}[Normal cones]
{\rm
Consider a set $Q\subset \h$. Then the {\em Fr\'{e}chet}, {\em limiting}, and {\em Clarke normal cones} to $Q$ at any point $\bar{x}\in\h$ are defined by $\hat{N}_Q(\bar{x}):=\hat{\partial} \delta(\bar{x})$, $N_Q(\bar{x}):=\partial \delta(\bar{x})$, and $N^c_Q(\bar{x}):=\partial_c \delta(\bar{x})$ respectively.}
\end{definition}
\smallskip

A particularly nice situation occurs when all the normal cones coincide.
\smallskip
\begin{definition}[Clarke regularity of sets]
{\rm A set $Q\subset\h$ is said to be {\em Clarke regular} at a point $\bar{x}\in Q$ if it is locally closed at $\bar{x}$ and every limiting normal vector to $Q$ at $\bar{x}$ is a Fr\'{e}chet normal vector, that is the equation $N_Q(\bar{x})=\hat{N}_Q(\bar{x})$ holds.}
\end{definition}
\smallskip

The functional version of Clarke regularity is as follows.
\smallskip
\begin{definition}[Subdifferential regularity]
{\rm A function $f\colon\h\to\overline{\R}$ is called {\em subdifferentially regular} at $\bar{x}$ if $f(\bar{x})$ is finite and $\mbox{\rm epi}\, f$ is Clarke regular at $(\bar{x},f(\bar{x}))$ as a subset of $\h\times\R$.}
\end{definition}
\smallskip

In particular, if $f\colon\h\to\overline{\R}$ is subdifferentially regular at a point $\bar{x}\in\dom f$, then equality 
$\hat{\partial} f(\bar{x})=\partial f(\bar{x})$ holds (\cite[Corollary 8.11]{VA}). Shortly, we will need the following result describing normals to sublevel sets \cite[Proposition 10.3]{VA}. We provide an independent proof for completeness and ease of reference in future work. The reader may safely skip it upon first reading. 
\smallskip
\begin{proposition}[Normals to sublevel sets]\label{prop:norm_level} Consider a lsc function $f\colon\h\to\overline{\R}$ and a point $\bar{x}\in\h$ with $0\notin\partial f(\bar{x})$.
Then the inclusion
$$
N_{[f\le f(\xb)]}(\xb)\subset (\cone \sd f(\xb))\cup\sd^{\infty}f(\xb)\quad \textrm{holds}.
$$
\end{proposition}

\begin{proof}
Define the real number $\alb:=f(\xb)$ and the sets $\cll_{\alb}:=\{ (x,\al):\; \al\le\alb\}$ and
$$
Q_{\alb}:=(\epi f)\bigcap \cll_{\alb}=\{(x,\al):\; f(x)\le\al\le\alb\}.
$$
We first show the implication 
\begin{equation}\label{one}
(x^*,0)\in N_{Q_{\alb}}(\xb,\alb)\  \Longrightarrow\  x^*\in (\cone \sd f(\xb))\bigcup\sd^{\infty} f(\xb).
\end{equation}
Indeed, consider a vector $(x^*,0)\in N_{Q_{\alb}}(\xb,\alb)$. 
Then Fuzzy calculus \cite[Chapter 2.1]{ioffe_survey} implies that there are sequences 
$(x_{1k},\al_{1k})\in \epi f$,
$(x^{*}_{1k},\beta_{1k})\in \hat{N}_{\sepi f}(x_{1k},\al_{1k})$, 
$(x_{2k},\al_{2k})\in \cll_{\alb}$
and $(x_{2k}^*,\beta_{2k})\in \hat{N}_{\cll_{\alb}}(x_{2k},\al_{2k})$ satisfying $$(x_{1k},\al_{1k})\to (\bar{x},\alb),\quad (x_{2k},\al_{2k})\to(\bar{x},\alb),\quad
x_{1k}^*+x_{2k}^*\to x^*, \quad\beta_{1k}+\beta_{2k}\to 0.$$
Observe $x_{2k}^*=0$ and hence $x^{*}_{1k}\to x^*$. Furthermore, by nature of epigraphs we have $\beta_{1k}\le 0$. If up to a subsequence we had $\beta_{1k}= 0$, then (\ref{one}) would follow immediately. Consequently, we may suppose that the inequality  $\beta_{1k}<0$ is valid. Then we have
$|\beta_{1k}|^{-1}x^{*}_{1k}\in\hat{\partial} f(x_{1k})$. Since the norms of $x_{1k}^*$ are uniformly bounded and we have $0\notin\sd f(\xb)$, the sequence $\beta_{1k}$
must be bounded. Consequently we may assume that $\beta_{1k}$  converges to some $\beta$ and (\ref{one}) follows.

Now consider a vector $u^*\in \hat{N}_{[f\le\alb]}(u)$ for some $u\in [f\le\alb]$. Consequently the inequality $\langle u^*,h\rangle\le o(\| h\|)$ holds, whenever $h$ satisfies
$f(u+h)\le\alb$. The latter in turn implies $(u^*,0)\in \hat{N}_{Q_{\alb}}(u,\alb)$.
Together with (\ref{one}),
taking limits of Fr\'{e}chet subgradients and applying equation (\ref{one}) completes the proof.    
\end{proof}
\smallskip

The following result, which follows from the proofs of \cite[Propositions 1 and 2, Chapter 3]{ioffe_survey}, establishes an elegant relationship between the slope and subdifferentials. 
\begin{proposition}[Slope and subdifferentials]\label{ferslope} 
Consider a lsc function $f\colon\h\to\overline{\R}$, and a point $\bar{x}\in\h$ with $f(\bar{x})$ finite.  Then we have $|\nabla f|(\bar{x})\leq d(0,\hat{\partial} f(\bar{x}))$, and furthermore 
the equality
$$
\slof (\bar{x})=d(0,\sd f(\bar{x})),\quad \textrm{ holds}.
$$ 
In particular, the two conditions $\slof (\bar{x})=0$ and $0\in\partial f(\bar{x})$ are equivalent.
\end{proposition}
\begin{proof}
The inequality $|\nabla f|(\bar{x})\leq d(0,\hat{\partial} f(\bar{x}))$ is immediate from the definition of the Fr\'{e}chet subdifferential. Now define $m= \slof (\bar{x})$.
One may easily check that if $m$ is infinite, then the subdifferential $\partial f(\bar{x})$ is empty, and therefore the result holds trivially. Consequently we may suppose that $m$ is finite. 

Fix an arbitrary $\epsilon >0$, and let $x$ be a point satisfying $$\|x-\bar{x}\|<\epsilon,\quad |f(x)-f(\bar{x})|<\epsilon,\quad \textrm{and }\quad|\nabla f|(x)< m+\epsilon.$$ 
Define the function $g(u):=f(u)+(m+\epsilon)\|u-x\|$. Observe that for all $u$ sufficiently close to $x$, we have $g(u)\geq f(x)$. We deduce (see e.g. \cite[Exercise 10.10]{VA}) $$0\in\partial g(x)\subset \partial f(x)+(m+\epsilon){\bf B}.$$ Hence we obtain the inequality $m+\epsilon\geq d(0,\partial f(x))$. Letting $\epsilon$ tend to zero, we deduce 
$m\geq d(0,\partial f(\bar{x}))$. 

To see the reverse inequality, consider a vector $\bar{v}\in\partial f(\bar{x})$ achieving $d(0,\partial f(\bar{x}))$. Then there exist sequences of points $x_i$ and vectors $v_i\in\hat{\partial} f(x_i)$ with $(x_i,f(x_i),v_i)\to(\bar{x},f(\bar{x}),\bar{v})$. Observe that for each index $i$, we have $\|v_i\|\geq |\nabla f|(x_i)$. Letting $i$ tend to infinity, the result follows.
\end{proof}
\smallskip


In particular, if $f$ is subdifferentially regular at $\bar{x}$, then the slope and the limiting slope are one and the same, that is the equation $\slof (\bar{x})=|\nabla f|(\bar{x})$ holds. We conclude this subsection with the following standard result of Linear Algebra.
\smallskip
\begin{lemma}[Result in Linear Algebra]\label{lem:lin}
Consider a subspace $V$ of $\h$. Then for any vector $b\in\h$, the equations 
$$P_V(b)=(b+V^{\perp})\cap V=\argmin_{z\in b+V^{\perp}}\|z\|, \quad \textrm{hold}.$$
\end{lemma}
\begin{proof}
Observe $b=P_V(b)+P_{V^{\perp}}(b)$, and consequently the inclusion $$P_V(b)\in(b+V^{\perp})\cap V\quad \textrm{holds}.$$ The reverse inclusion follows from the trivial computation 
$$z\in (b+V^{\perp})\cap V\Longrightarrow z-P_V(b)\in V \cap V^{\perp}\Longrightarrow z=P_V(b).$$

Now observe that first order optimality conditions imply that the unique minimizer $\bar{z}$ of the problem $$\min_{z\in b+V^{\perp}}\|z\|^2,$$ is characterized by the inclusion $\bar{z}\in (b+V^{\perp})\cap V$, and hence the result follows.
\end{proof}

\subsection{Main results}
In this section, we consider curves of near-maximal slope in Euclidean spaces. In this context, it is interesting to compare such curves to solutions $x\colon [0,\eta]\to\h$ of subgradient dynamical systems
$$\dot{x}(t)\in -\partial f(x(t)), \quad \textrm{ for a.e. } t\in [0,\eta].$$
It turns out that the same construction as in the proof of Theorem~\ref{thm:near_steep} shows that there exist near-steepest descent curves $x$ so that essentially, up to rescaling, the vector $\dot{x}(t)$ lies in $-\partial f(x(t))$ for a.e. $t\in [0,\eta]$. 

\smallskip
\begin{theorem}[Existence of near-steepest descent curves]
Consider a lsc function $f\colon\h\to\overline{\R}$, along with a point $\bar{x}$ in the domain of $f$. Suppose that $f$ is continuous on slope bounded sets.
Then there exists a curve of near-maximal slope $x\colon [0,L]\to\X$ emanating from $\bar{x}$ and satisfying 
$$\dot x(t)\in -\cl\cone\partial_c f(x(t)), \quad\textrm{ for a.e. } t\in[0,L].$$
\end{theorem}
\begin{proof}
We can clearly assume that zero is not a subgradient of $f$ at $\bar{x}$. We now specialize the construction of Theorem~\ref{thm:near_steep} to the Euclidean setting. Namely, we can find constants $\eta>0$, $r> 0$ and $C>0$ such that all conditions of Lemma \ref{blemma}
are satisfied for $\xb$,  $\al=f(\xb)-\eta$ and $K=C$. In particular, shrinking $\eta$ we may enforce the inequality $\eta<rC$. Let $0=\tau_0<\tau_1<\ldots<\tau_k=\eta$ be a partition of $[0,\eta]$ 
into $k$ equal parts. 
With this partition we can associate a piecewise linear function $u(\tau)$ as follows.
Set $u_k(0)=\xb$, and inductively define $u_k(\tau_{i+1})$ to be any point belonging to the projection of $u_k(\tau_i)$ onto the lower level set $[f\leq f(\bar{x})-\tau_{i+1}]$, provided that this set is nonempty. It is easy to see that these curves are admissible as the curves $u_k$ constructed in the proof of Theorem~\ref{thm:near_steep}. In particular, we deduce that up to a subsequence, $u_k$ converge uniformly to a Lipschitz curve $x$, and the derivative mappings $\|\dot{u}\|$ converge weakly to $\|\dot{x}\|$. 

Observe that in light of Proposition~\ref{prop:norm_level}, for any index $k$ and any $\tau\in [\tau_i,\tau_{i+1}]$ (for $i=1,\ldots, k$) we have $\dot{u}_k(\tau)\in -\big(\cone\partial f(u_k(\tau_{i+1}))\big)\cup \partial^{\infty} f(u_k(\tau_{i+1}))$.
Furthermore, recall that restricting to a subsequence we may suppose that $\dot{u}_k$ converges weakly to $\dot{x}(\tau)$ in $L^2(0,\eta)$. Mazur's Lemma then implies that a sequence of convex combinations of the form $\sum^{N(k)}_{n=k} \alpha^k_n \dot{u}_n$ converges strongly to $\dot{x}$ as $k$ tends to $\infty$. Since convergence in $L^2(0,\eta)$ implies almost everywhere pointwise convergence, we deduce that for almost every $\tau\in [0,\eta]$, we have $$\Big\|\sum^{N(k)}_{n=k} \alpha^k_n \dot{u}_n(\tau)-\dot{x}(\tau)\Big\|\to 0.$$ 
Therefore if the inclusion $$\dot{x}(\tau)\in -\cl\conv\Big[\big(\cone\, \partial f(x(\tau))\big)\cup\partial^{\infty} f(x(\tau))\Big]$$ did not hold, then we would deduce that there exists a subsequence of vectors $\dot{u}^{k_l}_{n_l}(\tau)$ with $\lim_{l\to\infty} \dot{u}^{k_l}_{n_l}(\tau)$ not lying in the set on the right-hand-side of the inclusion above. This immediately yields a contradiction. After the reparametrization performed in the proof of Theorem~\ref{thm:near_steep}, the curve $\gamma$ is subdifferentiable almost everywhere on $[0,L]$ and consequently satisfies
$$\dot{\gamma}(t)\in -\cl\cone \partial_c f(\gamma(t)), \quad \textrm{ for a.e. } t\in [0,L],$$
as we needed to show.
\end{proof}
\smallskip

The above theorem motivates the question of when curves of near-maximal slope and solutions of subgradient dynamical systems are one and the same, that is when is the rescaling of the gradient $\dot{x}$ in the previous theorem not needed. The following property turns out to be crucial.
\smallskip
\begin{definition}[Chain rule]
{\rm
Consider a lsc function $f\colon\R^n\to\overline{\R}$. We say that $\slof$ {\em admits a chain rule} if for every curve $x\in AC(a,b,\h)$ for which the composition $f\circ x$ is non-increasing and $f$ is subdifferentiable almost everywhere along $x$, the equation
$$(f\circ x)'(t)= \langle \partial f(x(t)),\dot{x}(t)\rangle\quad\quad \textrm{ holds for a.e. } t\in (a,b).$$}
\end{definition}
The following simple proposition shows that whenever $\slof$ admits a chain rule, solutions to subgradient dynamical systems and curves of near-maximal slope coincide.
\begin{proposition}[Subgradient systems \& curves of near-maximal slope]\label{prop:evol_max} {\\}
Consider a lsc function $f\colon\h\to\overline{\R}$ and suppose that $\slof$ admits a chain rule. 
Then for any curve $x\in AC(a,b,\h)$ the following are equivalent.
\begin{enumerate}
\item $x$ is a curve of near-maximal slope. \label{1_ev}
\item $f\circ x$ is nondecreasing and we have 
\begin{equation}\label{eqn:inc_evol_pre}
\dot{x}(t)\in-\partial f(x(t)), \quad \textrm{ a.e. on }[a,b].
\end{equation} \label{2_ev}
\item $f\circ x$ is nondecreasing and we have
\begin{equation}\label{eqn:inc_evol}
\dot{x}(t)\in-\partial f(x(t)), \quad {\textrm and } \quad\|\dot{x}(t)\|= d(0,\partial f(x(t))), \quad \textrm{ a.e. on }[a,b].
\end{equation}\label{3_ev}
\end{enumerate}
\end{proposition}
\begin{proof}
We first prove the implication $\ref{1_ev}\Rightarrow\ref{3_ev}$. 
To this end, suppose that $x$ is a curve of near-maximal slope. Then clearly $f\circ x$ is nondecreasing and we have
$$
\langle \partial f(x(t)),\dot{x}(t)\rangle=(f\circ x)'(t)\leq - \big(\slof(x(t))\big)^2 \quad   {\rm a.e. \  on} \;  [a,b].
$$
Let $v(t)\in \partial f(x(t))$ be a vector of minimal norm. Then we have 
$$\langle\partial f(x(t)),\dot{x}(t)\rangle\geq -\|v(t)\|\cdot \|\dot{x}(t)\|=-\big(\slof(x(t))\big)^2,$$
with equality if and only if $\dot{x}(t)$ and $v(t)$ are collinear. We deduce $\dot{x}(t)=-v(t)$, as claimed.

The implication $\ref{3_ev}\Rightarrow\ref{2_ev}$ is trivial. Hence we focus now on $\ref{2_ev}\Rightarrow\ref{1_ev}$. To this end suppose that $\ref{2_ev}$ holds and observe 
\begin{equation}\label{eqn:const}
\langle \partial f(x(t)),\dot{x}(t)\rangle =-\|\dot{x}(t)\|^2, \quad\textrm{ for a.e. } t\in (a,b).
\end{equation}
Given such $t$ consider the affine subspaces 
$$V=\para \partial f\big(x(t)),$$
Then we have $$\aff \partial f(x(t))=-\dot{x}(t)+V.$$
We claim now that the inclusion $\dot{x}(t)\in V^{\perp}$ holds. To see this, observe that for any real $\lambda_i$ and for vectors $v_i\in\partial f\big(x(t)\big)$, we have
$$\langle \dot{x}(t), \sum_{i=1}^k \lambda_i(v_i+\dot{x}(t))\rangle=\sum_{i=1}^k\lambda_i \big[\langle \dot{x}(t),v_i  \rangle+ \|\dot{x}(t)\|^2 \big]=0,$$
where the latter equality follows from (\ref{eqn:const}). Hence the inclusion $$-\dot{x}(t)\in (-\dot{x}(t) + V) \cap V^{\perp},$$ holds. Consequently, using Lemma~\ref{lem:lin}, we deduce that $-\dot{x}(t)$ achieves the distance of the affine space, $\aff \partial f(x(t))$, to the origin. On the other hand, the inclusion $-\dot{x}(t)\in \partial f \big(x(t)\big)$ holds, and hence $-\dot{x}(t)$ actually achieves the distance of $\partial f(x(t))$ to the origin. The result follows.
\end{proof}
\smallskip

In light of the theorem above, it is interesting to understand for which functions $f$ the slope $\slof$ admits a chain rule. Subdifferentially regular (in particular, all lsc convex) functions  furnish a simple example. The convex case can be found in \cite[Lemma 3.3, p 73]{Brezis}(Chain rule).
\smallskip 
\begin{lemma}[Chain rule under subdifferential regularity]\label{lem:dir}
Consider a subdifferentially regular function $f\colon\R^n\to\overline{\R}$. Then $\slof$ {\em admits a chain rule}.
\end{lemma}
\begin{proof}
Consider a curve $x\colon (a,b)\to\R^n$. Suppose that for some real $t\in (a,b)$ both $x$ and $f\circ x$ are differentiable at $t$ and $\partial f(x(t))$ is nonempty. We then deduce
\begin{align*}
\frac{d(f\circ x)}{dt}(t)&=\lim_{\epsilon\downarrow 0} \frac{f(x(t+\epsilon))-f(x(t))}{\epsilon}\\
&\geq \langle v,\dot{x}(t)\rangle,\quad \textrm{ for any } v\in\hat{\partial} f(x(t).
\end{align*}
Similarly we have 
\begin{align*}
\frac{d(f\circ x)}{dt}(t)&=\lim_{\epsilon\downarrow 0} \frac{f(x(t-\epsilon))-f(x(t))}{-\epsilon}\\
&\leq \langle v,\dot{x}(t)\rangle,\quad \textrm{ for any } v\in\hat{\partial} f(x(t)).
\end{align*}
Hence the equation
$$\frac{d(f\circ x)}{dt}(t)=\langle \hat{\partial} f(x(t))\dot{x}(t)\rangle\quad\textrm{ holds},$$ and the result follows.  
\end{proof}
\smallskip

Subdifferentially regular functions are very special, however. In particular, many nonpathological functions such as $-\|\cdot\|$ are not subdifferentially regular. So it is natural to consider prototypical nonpathological functions appearing often in practice --- those that are semi-algebraic. For an extensive discussion on semi-algebraic geometry, see the monographs of Basu-Pollack-Roy \cite{ARAG}, Lou van den Dries \cite{LVDB}, and Shiota \cite{Shiota}. For a quick survey, see the article of van den Dries-Miller \cite{DM} and the surveys of Coste \cite{Coste-semi, Coste-min}. Unless otherwise stated, we follow the notation of \cite{DM} and \cite{Coste-semi}.

A {\em semi-algebraic} set $S\subset\R^n$ is a finite union of sets of the form $$\{x\in \R^n: P_1(x)=0,\ldots,P_k(x)=0, Q_1(x)<0,\ldots, Q_l(x)<0\},$$ where $P_1,\ldots,P_k$ and $Q_1,\ldots,Q_l$ are polynomials in $n$ variables. In other words, $S$ is a union of finitely many sets, each defined by finitely many polynomial equalities and inequalities. A function $f\colon\R^n\to\overline{\R}$ is {\em semi-algebraic} if $\mbox{\rm epi}\, f\subset\R^{n+1}$ is a semi-algebraic set. 

Our goal now is to analyze the chain rule for the slope in the context of semi-algebraic functions. Before we proceed, we need to recall the notion of tangent cones.
\begin{definition}[Tangent cone]
{\rm Consider a set $Q\subset\R^n$ and a point $\bar{x}\in Q$. Then the {\em tangent cone} to $Q$ at $\bar{x}$, is simply the set 
$$T_Q(\bar{x}):=\Big\{\lim_{i\to\infty} \lambda_i (x_i-\bar{x}): \lambda_i\uparrow \infty \textrm{ and } x_i\in Q\Big\}.$$
}
\end{definition}

We now record the following simple lemma, whose importance in the context of semi-algebraic geometry will become apparent shortly. We omit the proof since it is rather standard.
\smallskip
\begin{lemma}[Generic tangency]\label{lem:touch}
Consider a set $M\subset\R^n$ and a path $x\colon [0,\eta]\to \R^n$ that is differentiable almost everywhere on $[0,\eta]$. Then for almost every $t\in [0,\eta]$, the implication
$$x(t)\in M \quad\Longrightarrow\quad \dot{x}(t)\in T_M(x(t)),\quad \textrm{ holds}.$$
\end{lemma}

The following is a key property of semi-algebraic functions that we will exploit \cite[Proposition 4]{Lewis-Clarke}. 
\smallskip
\begin{theorem}[Projection formula]\label{thm:proj}
Consider a lsc semi-algebraic function $f\colon\R^n\to\overline{\R}$. Then there exists a partition of $\dom f$ into finitely many ${\bf C}^1$-manifolds $\{M_i\}$ so that $f$ restricted to each manifold $M_i$ is ${\bf C}^1$-smooth. Moreover for any point $x$ lying in a manifold $M_i$, the inclusion $$\partial_c f(x)\subset \nabla g(x) +N_{M_i}(x) \quad\textrm{ holds},$$
where $g\colon\R^n\to\R$ is any ${\bf C}^1$-smooth function agreeing with $f$ on a neighborhood of $x$ in $M_i$. 
\end{theorem}
\smallskip

\begin{theorem}[Semi-algebraic chain rule for the slope]
Consider a lsc semi-algebraic function $f\colon\R^n\to\overline{\R}$ that is locally Lipschitz continuous on its domain. Consider also a curve $\gamma\in AC(a,b,\h)$ whose image is contained in the domain of $f$. Then equality
$$(f\circ\gamma)'(t)= \langle \partial f(\gamma(t)),\dot{x}(t)\rangle=\langle \partial_c f(\gamma(t)),\dot{x}(t)\rangle$$
holds for almost every $t\in [a,b]$.
In particular, the slope $\slof$ admits a chain rule.
\end{theorem}
\begin{proof}
Consider the partition of $\dom f$ into finitely many ${\bf C}^1$-manifolds $\{M_i\}$, guaranteed to exist by Theorem~\ref{thm:proj}. 
We first record some preliminary observations. Clearly both $x$ and $f\circ x$ are differentiable at a.e. $t\in (0,T)$. Furthermore, in light of Lemma~\ref{lem:touch}, for any index $i$ and for a.e. $t\in (0,\eta)$ the implication
$$x(t)\in M_i\quad\Longrightarrow \quad\dot{x}(t)\in T_{M_i}(x(t)),\quad \textrm{ holds}.$$
Now suppose that for such $t$, the point $x(t)$ lies in a manifold $M_i$ and let $g\colon\R^n\to\R$ be a ${\bf C}^1$-smooth function agreeing with $f$ on a neighborhood of $x(t)$ in $M_i$. Lipschitzness of $f$ on its domain then easily implies 
\begin{align*}
\frac{d(f\circ x)}{dt}(t)&=\lim_{\epsilon\downarrow 0} \frac{f(x(t+\epsilon))-f(x(t))}{\epsilon}\\
&=\lim_{\epsilon\downarrow 0} \frac{f(P_{M_i}(x(t+\epsilon)))-f(x(t))}{\epsilon}\\
&=\lim_{\epsilon\downarrow 0} \frac{g(P_{M_i}(x(t+\epsilon)))-g(x(t))}{\epsilon}\\
&=\frac{d}{dt}\, g\circ P_{M_i}\circ x(t)=\langle \nabla g(x(t)),\dot{x}(t)\rangle\\
&=\langle \nabla g(x(t))+N_{M_i}(x(t)),\dot{x}(t)\rangle.
\end{align*}
The result follows.
\end{proof}
\smallskip

A noteworthy point about the theorem above is the appearance of the Clarke subdifferential in the chain rule. As a result,  we can strengthen Theorem~\ref{prop:evol_max} in the context of lsc semi-algebraic functions $f\colon\R^n\to\overline{\R}$ that are locally Lipschitz continuous on their domains. The proof is analogous to that of Theorem~\ref{prop:evol_max}.
\smallskip

\begin{proposition}[Semi-algebraic equivalence]
Consider a lsc semi-algebraic function $f\colon\h\to\overline{\R}$ that is locally Lipschitz continuous on its domain.
Then for any curve $x\in AC(a,b,\h)$ the following are equivalent.
\begin{enumerate}
\item $x$ is a curve of near-maximal slope. 
\item $f\circ x$ is nondecreasing and we have 
\begin{equation*}
\dot{x}\in-\partial f(x), \quad \textrm{ a.e. on }[a,b].
\end{equation*} 
\item $f\circ x$ is nondecreasing and we have
\begin{equation*}
\dot{x}\in-\partial f(x), \quad\|\dot{x}\|= d(0,\partial f(x)), \quad {\textrm and }\quad \|\dot{x}\|= d(0,\partial_c f(x)) \quad \textrm{ a.e. on }[a,b].
\end{equation*}
\end{enumerate}
\end{proposition}

We end this section by showing that for semi-algebraic functions, bounded curves of near-maximal slope necessarily have bounded length. The argument we present is not new; rather we include this discussion with the purpose of painting a more complete picture for the reader, illustrating that semi-algebraic functions provide an appealing setting for the analysis of steepest descent curves. We begin with the celebrated Kurdyka-{\L}ojasiewicz inequality.
\smallskip
\begin{definition}[Kurdyka-{\L}ojasiewicz inequality]
{\rm
A function $f\colon\R^n\to\overline{\R}$ is said to satisfy the {\em Kurdyka-{\L}ojasiewicz inequality} if for any bounded open  set $U\subset\R^n$ and any real $\tau$, there exists $\rho >0$ and a non-negative continuous function $\psi\colon [\tau,\tau+\rho)\to\R$, which is ${\bf C}^1$-smooth and strictly increasing on $(\tau,\tau+\rho)$, and such that the inequality 
$$|\nabla(\psi\circ f)|(x)\geq 1,$$
holds for all $x\in U$ with $\tau<f(x)<\tau+\rho$.
}
\end{definition}
\smallskip

In particular, all semi-algebraic functions satisfy the Kurdyka-{\L}ojasiewicz inequality \cite{tame_opt}. For an extensive study the Kurdyka-{\L}ojasiewicz inequality and a description of its historical significance, see for example \cite{tailwag}. The proof of the following theorem is almost identical to the proof of \cite[Theorem 7.1]{tame_opt}; hence we only provide a sketch. In fact, the theorem remains valid if rather than assuming semi-algebraicity, we only assume that the Kurdyka-{\L}ojasiewicz inequality is satisfied.
\smallskip
\begin{theorem}[Lengths of curves of near-maximal slope]{ \\}
Consider a lsc, semi-algebraic function $f\colon\R^n\to\overline{\R}$, and let $U$ be a bounded subset of $\R^n$. Then there exists a number $N >0$ such that the length of any curve of near-maximal slope for $f$ lying in $U$ does not exceed $N$.
\end{theorem}
\begin{proof}
Let $x\colon [0,T)$ be a curve of near-maximal slope for $f$ and let $\psi$ be any strictly increasing ${\bf C}^1$-smooth function on an interval containing the image of $f\circ x$. It is then easy to see then that, up to a reparametrization, $x$ is a curve of near-maximal slope for the composite function $\psi\circ f$. In particular, we may assume that $f$ is bounded on $U$, since otherwise we may for example replace $f$ by $\psi\circ f$ where $\psi(t)=\frac{t}{\sqrt{1+t^2}}$. 

Define the function 
$$\xi(s)=\inf\{|\nabla f|(x): x\in U,\, f(x)=s\}.$$

Standard arguments show that $\xi$ is semi-algebraic. Consequently, with an exception of finitely many points, the domain of $\xi$ is a union of finitely many open intervals $(\alpha_i,\beta_i)$, with $\xi$ continuous and either strictly monotone or constant on each such interval. 
Define for each index $i$, the quantity $$c_i=\inf\{\xi(s):s\in (\alpha_i,\beta_i)\}.$$ We first claim that $\xi$ is strictly positive on each interval $(\alpha_i,\beta_i)$. This is clear for indices $i$ with $c_i >0$. On the other hand if we have $c_i=0$, then by Sard's theorem \cite{ioffe_strat} the function $\xi$  is strictly positive on $(\alpha_i,\beta_i)$ as well.

Define $\zeta_i$ and $\eta_i$ by $$\zeta=\inf\{t: f(x(t))=\alpha_i\}\quad \textrm{and} \quad \eta=\sup\{t: f(x(t))=\beta_i\},$$
and let $l_i$ be the length of $x(t)$ between $\zeta_i$ and $\eta_i$.

Then we have
$$l_i=\int_{\zeta_i}^{\eta_i}\|\dot{x}(t)\| dt=\int_{\zeta_i}^{\eta_i}\slof (x(t)) dt\leq\Big((\eta_i-\zeta_i)\int_{\zeta_i}^{\eta_i}\slof(x(t))^2 dt\Big)^{\frac{1}{2}}.$$
On the other hand, observe
$$\int_{\zeta_i}^{\eta_i}\slof(x(t))^2 dt=f(x(\eta_i))-f(x(\zeta_i))=\beta_i-\alpha_i.$$
Finally in the case $c_i >0$ we have $l_i\geq c_i(\eta_i-\zeta_i)$, which combined with the two equations above yields the bound $$l_i\leq\frac{\beta_i-\alpha_i}{c_i}.$$

If the equation $c_i=0$ holds, then by the Kurdyka-{\L}ojasiewicz inequality we can find a continuous function $\xi_i\colon [\alpha_i,\alpha_i+\rho)\to\R$, for some $\rho >0$, where $\xi$ is strictly positive and ${\bf C}^1$-smooth on $(\alpha_i,\alpha_i+\rho)$ and satisfying 
$|\nabla (\xi_i\circ f)|(y)\geq 1$
for any $y\in U$ with $\alpha_i<f(y)<\alpha_i+\rho$.
Since $\xi_i$ is strictly increasing on $(\alpha_i,\alpha_i+\rho)$, it is not difficult to check 
that we may extend $\xi_i$ to a continuous function on $[\alpha_i,\beta_i]$ and so that this extension is ${\bf C}^1$-smooth and strictly increasing on $(\alpha_i,\beta_i)$ with the inequality $|\nabla (\xi_i\circ f)|(y)\geq 1$ being valid for any $y\in U$ with $\alpha_i<f(y)<\beta_i$.

Then as we have seen before, up to a reparametrization, the curve $x(t)$ for $t\in [\zeta_i,\eta_i]$ is a curve of near maximal slope for the function $\xi_i\circ f$. Then as above, we obtain the bound $l_i\leq \xi_i(\beta_i)-\xi_i(\alpha_i)$. 

We conclude that the length of the curve $x(t)$ is bounded by a constant that depends only on $f$ and on $U$, thereby completing the proof.
\end{proof}
\smallskip

The following consequence is now immediate.
\smallskip
\begin{corollary}[Convergence of curves of near-maximal slope]
Consider a lsc, semi-algebraic function $f\colon\R^n\to\overline{\R}$. Then any curve of near-maximal slope for $f$ that is bounded and has a maximal domain of definition converges to a lower-critical point of $f$.
\end{corollary}

\bibliographystyle{plain}
\small
\parsep 0pt
\bibliography{trajectories}

\end{document}